\documentclass{amsart}
\usepackage[demo]{graphicx}
\usepackage{caption}
\usepackage{subcaption}
\usepackage{graphicx, amsmath, amsfonts, epsfig, latexsym, amssymb, amsthm}
\usepackage{tikz}
\usepackage{tkz-berge}
\usepackage{float}

\usetikzlibrary {positioning}
\definecolor {processblue}{cmyk}{0.96,0,0,0}
\newtheorem{thm}{Theorem}[section]
\theoremstyle{definition}
\newtheorem{cor}[thm]{Corollary}
\newtheorem{prop}[thm]{Proposition}
\newtheorem{defn}[thm]{Definition}

\newtheorem{rem}[thm]{Remark}
\newtheorem{ex}[thm]{Example}

\numberwithin{equation}{section}
\begin{document}
\title[Some graphs related to submodules of a module]
{Some graphs related to submodules of a module}
\author{Faranak Farshadifar}
\address{Department of Mathematics Education, Farhangian University, P.O. Box 14665-889, Tehran, Iran.}
\email{f.farshadifar@cfu.ac.ir}


\begin{abstract}
Let $R$ be a commutative ring with
identity and $M$ be an $R$-module.
In this paper, we introduce and investigate the second submodule intersection graph $SSI(M)$ of $M$ with vertices
are non-zero proper submodules of $M$ and two distinct vertices $N$ and $K$ are adjacent if and only if $N \cap K$ is a second submodule of $M$. Also, we introduce and consider the prime submodule sum graph $PSS(M)$ of $M$ with vertices
are non-zero proper submodules of $M$ and two distinct vertices $N$ and $K$ are adjacent if and only if $N + K$ is a prime submodule of $M$.
\end{abstract}
\maketitle
\section{Introduction}
\noindent
Throughout this paper, $R$ will denote a commutative ring with identity and $\Bbb Z$ will denote the ring of of integers.
Also, "$\subset$" will denote the strict inclusion.

A \emph{graph} $G$ is defined as the pair $(V(G),E(G))$, where $V(G)$ is the set of vertices of $G$ and $E(G)$ is the set of edges of $G$. For two distinct vertices $a$ and $b$ of $V(G)$, the notation $a-b$ means that $a$ and $b$ are adjacent. A graph $G$ is said to be \emph{complete} if $a-b$ for all distinct $a, b\in V(G)$, and $G$ is said to be \emph{empty} if $E(G) =\emptyset$. Note by this definition that a graph may be empty even if $V (G)\not =\emptyset$. An empty graph could also be described as totally disconnected. If $|V (G)|\geq 2$, a \emph{path} from $a$ to $b$ is a series
of adjacent vertices $a-v_1-v_2-...-v_n-b$. The \emph{length of a path} is the number of edges it contains.
A graph $G$ is \emph{connected} if for every pair of distinct vertices $a, b\in V (G)$, there exists a path from $a$ to $b$. If there is a path from $a$ to $b$ with $a, b \in V (G)$,
then the \emph{distance from} $a$ to $b$ is the length of the shortest path from $a$ to $b$ and is denoted by $d(a, b)$. If there is not a path between $a$ and $b$, $d(a, b) = \infty$. The \emph{diameter} of $G$ is diam$(G) = Sup\{d(a,b)| a, b \in V(G)\}$.
A \emph{cycle} is a path that begins and ends at the same vertex in which no edge is repeated, and all vertices other than the starting and ending vertex are distinct. If a graph $G$ has a cycle, the \emph{girth} of $G$ (notated $g(G)$) is defined as the length of the shortest cycle of $G$; otherwise, $g(G) =\infty$.
A graph without any cycles is an \textit{acyclic graph}. A vertex that is adjacent to
every other vertices is said to be a\textit{universal vertex} whereas a vertex with degree
zero is called an \textit{isolated vertex}.

A proper submodule $P$ of an $R$-module $M$ is said to be \textit{prime} if for any
$r \in R$ and $m \in M$ with $rm \in P$, we have $m \in P$ or $r \in (P :_R M)$ \cite{MR183747, Da78}.
A non-zero submodule $S$ of an $R$-module $M$ is said to be \emph{second} if for every element $r$ of $R$ we have either $rS= 0$ or $rS= S$ \cite{Y01}. More information about this class of modules can be found in \cite{FF1404}. Also, an ideal $I$ of $R$ is said to be a \textit{second ideal} if $I$ is a second submodule of the $R$-module $R$.

In \cite{saha2023prime}, the authors introduced and investigate the definition of the \textit{prime ideal sum graph} of $R$, denoted
by $PIS(R)$, which is a graph whose vertices are non-zero proper ideals of $R$ and two distinct
vertices $I$ and $J$ are adjacent if and only if $I + J$ is a prime ideal of $R$. Also, in \cite{FF1405}, the present author
defined and studied the \textit{second ideal intersection graph} of $R$, denoted
by $SII(R)$, which is a graph whose vertices are non-zero proper ideals of $R$ and two distinct
vertices $I$ and $J$ are adjacent if and only if $I \cap J$ is a second ideal of $R$.

Let $M$ be an $R$-module. Motivated by previous studies on the prime ideal sum graph and second ideal intersection graph of the ring $R$, in this paper, we introduce and investigate the second submodule intersection graph $SSI(M)$ of $M$ with vertices
are non-zero proper submodules of $M$ and two distinct vertices $N$ and $K$ are adjacent if and only if $N \cap K$ is a second submodule of $M$. Also, we introduce and consider the prime submodule sum graph $PSS(M)$ of $M$ with vertices
are non-zero proper submodules of $M$ and two distinct vertices $N$ and $K$ are adjacent if and only if $N + K$ is a prime submodule of $M$.
\section{Second submodule intersection graph of a module}
\noindent
\begin{defn}\label{2.1}
Let $M$ be an $R$-module.
The \textit{second submodule intersection graph} of $M$, denoted
by $SSI(M)$, is an undirected simple graph whose vertices are non-zero proper submodules of $M$ and two distinct
vertices $N$ and $K$ are adjacent if and only if $N \cap K$ is a second submodule of $M$.
\end{defn}

\begin{thm}\label{2.2}
Let $M$ be an $R$-module such that every submodule contains a minimal submodule. Then
$SSI(M)$ has a universal vertex if and only if one
of the two statements hold:
\begin{itemize}
\item [(a)] $M$ has exactly one minimal submodule.
\item [(b)] $M$ has exactly two minimal submodules $M_1$ and $M_2$ such that $M_1 + M_2$ is a
maximal submodule and that there is no non-second submodule that properly contained in $M_1+M_2$.
\end{itemize}
\end{thm}
\begin{proof}
Let (a) hold and $K$ be the minimal submodule of $M$. Then for each submodule $N$ of $M$, we have $N \cap K = K$, which is a second submodule and hence $N$ is adjacent to $K$. Thus, $K$ is a universal vertex.

Let (b) hold and set $N:= M_1 + M_2$. Assume that $H$ is a non-trivial submodule other than $N$
and without loss of generality, let $M_1\subseteq H$. Since by assumption, $N$ is a maximal submodule, we have
$M_1\subseteq N \cap H\subset N$. Now, by assumption, $N\cap H$ is a second submodule of $M$ and so $N$ is a universal vertex.

Conversely, let $SSI(M)$ have a universal vertex, say $N$. If $M$ has a unique minimal
submodule, the proof is done. Now, assume that $M$ has at least three minimal submodules, say
$M_1,M_2$ and $M_3$. Note that $N$ cannot be a minimal submodule as two distinct minimal
submodules are not adjacent. Since $N$ is not a minimal submodule, then it is contains a minimal submodule, say $M_1$. If possible, let $M_2 \not \subseteq N$. Then by the minimality of $M_2$, we have $N \cap M_2= 0$. Thus $N$ is not adjacent to $M_2$, a contradiction. Hence, $M_1+M_2+M_3\subseteq N$. Now, one of the following two cases holds.

\textbf{Case 1.} Let $N \not= M_1 +M_2+M_3$. Since $N\cap (M_1 +M_2+M_3)=M_1 +M_2+M_3$
and $N$ is a universal vertex in $SSI(M)$, we get that $M_1 +M_2+M_3$ is a second submodule of $M$.
Thus $M_1 +M_2+M_3\subseteq (0:_MAnn_R(M_1)Ann_R(M_2)Ann_R(M_3))$ implies that $M_1=M_2=M_3$, which is a contradiction.

\textbf{Case 2.} Let $N = M_1 +M_2+M_3$.  Set $T = M_1+M_2$. Since sum of two minimal
submodules cannot be a second submodule, $N \cap T = T$ is not a second submodule, i.e.  $N$ is not adjacent to $T$, which contradicts
with our assumption that $N$ is a universal vertex. Therefore, $M$ has exactly two
minimal submodules, say $M_1$ and $M_2$. By the same argument as above, we conclude that
$N = M_1 +M_2$. Now, we show that $N$ is a maximal submodule. If possible, let there exists a non-trivial
submodule $K$ of $M$ such that $M_1 +M_2=N\subset K$. Then, as $N$ is a universal vertex, $K \cap N = N = M_1 +M_2$ is a
second submodule of $M$, which is a contradiction. Thus, $N$ is a maximal submodule.
If possible, let $K$ be a non-second submodule such that $M_i \subset K \subset N = M_1 +M_2$, where
$i = 1$ or $2$. But, it follows that $N \cap K = K$, a non-second submodule and hence $N$ is not adjacent to $K$, a
contradiction. Thus, there does not exist such submodule $K$ and the proof is completed.
\end{proof}

An $R$-module $M$ is said to be \textit{coreduced} if $rM=r^2M$ for each $r\in  R$ \cite[Theorem 2.13]{MR3755273}.

For a submodule $N$ of an $R$-module $M$, the \emph{second radical} (or \emph{second socle}) of $N$ is defined  as the sum of all second submodules of $M$, contained in $N$, and it is denoted by $sec(N)$ (or $soc(N)$). In case $N$ does not contain any second submodule, the second radical of $N$ is defined to be $(0)$ \cite{MR3085034, AF11}. If $sec(M)=M$, then $M$ is coreduced \cite[Proposition 2.22]{MR3755273}.

\begin{rem}\label{2.3}
Let $M$ be an $R$-module such that every submodule of $M$ contains a minimal submodule and $M$ is not a coreduced module.
Then $sec(M)$ is adjacent to each second submodule of $M$.
\end{rem}
\begin{proof}
As $M$ is an $R$-module such that every submodule of $M$ contains a minimal submodule and $M$ is not a coreduced module, $sec(M)$ is a non-zero proper submodule of $M$. For each second submodule $S$ of $M$, we have $S\cap sec(M)=S$ is a second submodule of $M$. Hence,  $sec(M)$ is adjacent to each second submodule of $M$.
\end{proof}

\begin{cor}\label{2.4}
Let $M$ be an $R$-module such that every submodule of $M$ contains a minimal submodule and $M$ is not a coreduced module.
Then $sec(M)$ is the only minimal submodule of $M$ if and only if $sec(M)$ is a universal vertex of $SSI(M)$.
\end{cor}
\begin{proof}
This follows from Theorem \ref{2.2}.
\end{proof}

\begin{thm}\label{2.5}
Let $M$ be an $R$-module such that every submodule of $M$ contains a minimal submodule.
A submodule $N$ of $M$ is an isolated vertex in $SSI(M)$ if and only if $N$ is
a minimal and maximal submodule of $M$.
\end{thm}
\begin{proof}
Let $N$ be an isolated vertex in $SSI(M)$. If $N$ is not a minimal submodule, then it is
properly contains a minimal submodule, say $K$ and $N \cap K = K$. Thus $N$ is adjacent to $K$ in $SSI(M)$,
a contradiction, and so $N$ is a minimal submodule. If $N$ is not a maximal submodule, then there
exists a submodule $T$ of $M$ such that $N \subset T \subset M$ and $N \cap T =N$, which is a
second submodule since every minimal submodule of $M$ is a second submodule of $M$. Hence $N$ is adjacent to $T$, a contradiction, and so $N$ is a maximal submodule of $M$.

Conversely, let $N$ be a minimal and maximal submodule of $M$. If possible, assume
that $N$ is not isolated in $SSI(M)$. Then there exists a non-zero proper submodule $T$ of $M$ other than $N$ such that $N \cap T$ is a second submodule. If $N \not \subseteq T$, then as $N$ is minimal submodule of $M$, we
have $N\cap T =0$, which is not a second submodule. On the other hand if $N \subseteq T$, then by maximality
of $N$, we have $T =M$ or $N = T$, a contradiction. Thus such a submodule $T$ does not
exist and hence $N$ is an isolated vertex in $SSI(M)$.
\end{proof}

\begin{thm}\label{2.6}
Let $M$ be an $R$-module such that every submodule of $M$ contains a minimal submodule. Then
we have the following.
\begin{itemize}
\item [(a)] If $SSI(M)$ is a complete graph, then $M$ has exactly one minimal submodule and every non-zero
non-second submodule is a maximal submodule.
\item [(b)] If $M$ has exactly one minimal submodule, then $SSI(M)$ is a complete graph.
\end{itemize}
\end{thm}
\begin{proof}
(a) Let $SSI(M)$ be a complete graph. Then by Theorem \ref{2.2}, $SSI(M)$ has a universal vertex and
hence either one of the two conditions holds. Since two distinct minimal
submodules cannot be adjacent, part (a) of
Theorem \ref{2.2} holds. Let $N$ be a non-zero proper submodule of $M$ which is not second. If
possible, there exists a submodule $T$ of $M$ such that $N \subset T \subset M$, then $N \cap T = N$. Since
$N$ is not a second submodule, $N$ is not adjacent to $T$, a contradiction to the completeness of $SSI(M)$. Thus
every non-zero non-second submodule is a maximal submodule.

(b) Conversely, let $M$ has exactly one minimal submodule, say $K$. Since by assumption, every submodule of $M$ contains a minimal submodule, we have $K$ is contained in every non-zero submodule of $M$. Now, as every minimal submodule of $M$ is a second submodule of $M$, we have $SSI(M)$ is a complete graph.
\end{proof}

An $R$-module $M$ is said to be a \emph{comultiplication module} if for every submodule $N$ of $M$ there exists an ideal $I$ of $R$ such that $N=(0:_MI)$ \cite{MR3934877}.
\begin{prop}\label{2.7}
Let $M$ be a comultiplication $R$-module. Then two distinct non-trivial submodules $N$ and $K$ of $M$ with $Ann_R(N\cap K)=Ann_R(N)
+Ann_R(K)$ are adjacent in $SSI(M)$ if and only if $Ann_R(N)$ and $Ann_R(K)$ are adjacent in $PIS(R)$.
\end{prop}
\begin{proof}
This follows from the fact that in the comultiplication $R$-module $M$, we have $S$ is a second submodule of $M$ if and only if
$Ann_R(S)$ be a prime ideal of $R$ \cite[Theorem 196 (a)]{MR3934877}.
\end{proof}

\begin{ex}\label{2.8}
For each prime number $p$ the $\Bbb Z$-module $\Bbb Z_{p^\infty}$ is a star graph with center vertex $(1/p+\Bbb Z)\Bbb Z_{p^\infty}$.
\end{ex}

\begin{thm}\label{2.7}
Let $M$ be an $R$-module such that every submodule of $M$ contains a minimal submodule. Then $SSI(M)$ is connected if and only if $M$ is not a
direct sum of two of its minimal submodules. If $SSI(M)$ is connected, then diam$(SSI(M)) \leq 2$.
\end{thm}
\begin{proof}
Let $N_1$, $N_2$ be two non-trivial submodules of $M$. If $N_1\cap N_2$ is a second submodule, then $N_1- N_2$
in $SSI(M)$. Otherwise, assume that $N_1\cap N_2$ is not a second submodule. By assumption, both submodules $N_1$ and $N_2$ are contain a minimal submodules of $M$. If they are contain the same minimal submodule, say $K$, then we have $N_1 - K - N_2$, as minimal submodules
are second and hence $d(N_1, N_2) = 2$. Thus, we assume that they are not contain the same minimal submodules, $K_1\subset N_1$, $K_2\subset N_2$ and $K_1 \not= K_2$ are two minimal submodules of $M$.
First assume that $K_1+K_2 \not=M$. By the minimality of $K_2$, we have $K_2 \cap N_1=0$.
Then $N_1 \cap (K_1+K_2)=N_1\cap K_1+N_1\cap K_2=K_1$ and so $N_1 -  K_1 + K_2$. Similarly,
 $K_1 + K_2-N_2$. Thus we have a path of length 2 given
by $N_1 -  K_1 + K_2 - N_2$ and so  $d(N_1, N_2) \leq 2$. Hence,  diam$(SSI(M)) \leq 2$. If $K_1+ K_2 =M$, then since $K_1$ and $K_2$ are minimal submodules, the sum is directed.

Conversely, if $M$ is direct sum of two minimal submodules $K_1$ and $K_2$, then
$K_1\cong M/K_2$ and $K_2\cong M/K_1$. Thus $K_1$ and $K_2$ are also maximal submodules of $M$. This in turn implies that
$SSI(M)$ have two isolated vertices $K_1$ and $K_2$. Thus $SSI(M)$ is not connected.
\end{proof}

\begin{cor}\label{2.8}
Let $M$ be a comultiplication $R$-module such that $M$ is not a
direct sum of two of its minimal submodules. Then diam$(SSI(M)) \leq 2$.
\end{cor}
\begin{proof}
Since every submodule of $M$ contains a minimal submodule by \cite[Theorem 7]{MR3934877}, the result follows from Theorem \ref{2.7}.
\end{proof}

\begin{cor}\label{2.998}
Let $M$ be an $R$-module such that every submodule of $M$ contains a minimal submodule. If $SSI(M)$ is a connected graph, then the intersection of any two maximal submodules of $M$ is non-zero.
\end{cor}
\begin{proof}
Let $SSI(M)$ be a connected graph.
Assume contrary that $K_1$ and $K_2$ are two maximal submodules of $M$ such that $K_1 \cap K_2=0$.
Since every submodule of $M$ contains a minimal submodule, there is a minimal submodule $S_1$ of $M$ such that $S_1 \subseteq K_1$.
Since  $K_1 \cap K_2=0$, $S_1 \not \subseteq K_2$. Thus as $K_2$ is maximal, $K_2+S_1=M$. Hence
$$
K_1=K_1\cap M=K_1 \cap (S_1+K_2)=K_1 \cap S_1+K_1\cap K_2=S_1.
$$
Thus $K_1$ is a maximal and minimal submodule of $M$. Hence by Theorem \ref{2.5}, $K_1$ is an isolated vertex, which is a contradiction.
\end{proof}

\begin{thm}\label{2.9}
Let $M$ be an $R$-module such that $M$ have two non-comparable submodules which are adjacent in $SSI(M)$. Then
girth$(SSI(M)) = 3$.
\end{thm}
\begin{proof}
Let $N_1$ and $N_2$ be two non-comparable submodules which are adjacent in $SSI(M)$.
Then $N_1\cap N_2$ is a second submodule of $M$. Since $N_1$ and $N_2$ are non-comparable, then $N_1$, $N_2$, and $N_1\cap N_2$ forms a triangle in $SSI(M)$, i.e. girth($SSI(M)$) = 3.
\end{proof}

\begin{cor}\label{2.10}
If $SSI(M)$ is acyclic or girth$(SSI(M)) > 3$, then no two non-comparable
submodules of $M$ are adjacent in $SSI(M)$ and adjacency occurs only in case
of comparable submodules, i.e. for any edge in $SSI(M)$, one of the terminal vertices is a
second submodule of $M$.
\end{cor}

\begin{thm}\label{2.11}
Let $M$ be an $R$-module such that girth($SSI(M)) = n$. Then there exist at least
$\lfloor n/2 \rfloor$ distinct second submodules in $M$.
\end{thm}
\begin{proof}
By Theorem \ref{2.9}, if two non-comparable submodules are adjacent in $SSI(M)$,
then girth($SSI(M) = 3$ and the intersection of those two non-comparable submodules forms a
second submodule, and hence $M$ contains at least
$\lfloor 3/2 \rfloor=1$ second submodule. Thus, we assume
that girth$(SSI(M)) > 3$, i.e. by Corollary \ref{2.10}, adjacency occurs only in case of
comparable submodules. Let $N_1 - N_2 -N_3- \cdots - N_n - N_1$ be a cycle of length $n$. First, we
observe that neither $N_1 \subseteq N_2 \subseteq  N_3\subseteq \cdots \subseteq N_n \subseteq N_1$ nor $N_1 \supseteq N_2 \supseteq N_3\supseteq \cdots\supseteq N_n \supseteq N_1$ can hold, as in both the cases all the submodules will be equal. Thus, without loss of
generality, we have $N_2 \subseteq  N_1, N_3$ and $N_2$ is a second submodule of $M$. Hence, we have the following
two cases:

\textbf{Case I.} Let $N_2 \subseteq  N_1, N_3$ and $N_4 \subseteq N_3$. Then, we have $N_2$, $N_4$ to be second submodules.

\textbf{Case II.} Let $N_2 \subseteq  N_1, N_3$ and $N_3 \subseteq N_4$. Then, we have $N_2$, $N_3$ to be second submodules.
In any case, we get at least 2 submodules to be second in $M$ among $N_1$, $N_2$, $N_3$ and $N_4$.
Continuing in this manner till $N_n$, we get at least
$\lfloor n/2 \rfloor$ submodules which are second in $M$.
\end{proof}

\begin{cor}\label{2.12}
If an $R$-module $M$ has $k$ second submodules, then $SSI(M)$ is either acyclic or
girth$(SSI(M)) \leq 2k$.
\end{cor}

A submodule $N$ of an $R$-module $M$ is called \textit{large submodule} in $M$ if for every non-zero submodule $K$ of $M$ we have $N \cap K \not=0$ \cite{AF74}. An $R$-module $M$ is called \textit{uniform module} if every non-zero submodule of $M$ is large submodule in $M$ \cite{Goo76}.
\begin{prop}\label{2.8}
Let $M$ be a uniform $R$-module such that every non-zero submodule of $M$ is second. Then $SSI(M)$ is a complete garph.
\end{prop}
\begin{proof}
This is straightforward
\end{proof}

\begin{defn}
We defined the subgraph $\widetilde{PSS(M)}$ of the graph $PIS(R)$ whose vertices are $(N:_RM)$ for all submodules $N$ of $M$ such that $(N:_RM)$ is a non-trivial ideal of $R$ and
 two distinct vertices $(N:_RM)$ and $(K:_RM)$ are adjacent if and only if $(N:_RM)+(K:_RM)$ is a prime ideal of $R$.
 \end{defn}

An $R$-module $M$ satisfies the \emph{double annihilator
conditions} (DAC for short)  if for each ideal $I$ of $R$
we have $I=Ann_R((0:_MI))$.
An $R$-module $M$ is said to be a \emph{strong comultiplication module} if $M$ is
a comultiplication $R$-module and satisfies the DAC conditions \cite{MR3934877}.

\begin{thm}\label{21.9}
Let $M$ be an $R$-module. If $M$ is a strong comultiplication $R$-module such that $SSI(M)$ is a complete garph, then $\widetilde{PSS(M)}$ is a complete graph.
\end{thm}
\begin{proof}
Let $M$ be a strong comultiplication $R$-module and $SSI(M)$ be a complete garph. Assume that $N$ and $K$ are two submodules of $M$ such that $(N:_RM)$ and $(K:_RM)$ are two non-trivial ideals of $R$. Then $(0:_M(N:_RM))$ and $(0:_M(K:_RM))$ are two non-trivial submodules of $M$. So $(0:_M(N:_RM))\cap (0:_M(K:_RM))$ is a second submodule of $M$.
Hence $Ann_R((0:_M(N:_RM))\cap (0:_M(K:_RM)))= (N:_RM)+(K:_RM)$ is a prime ideal of $R$.
\end{proof}

Let $G$ be a graph. A non-empty subset $D$ of the
vertex set $V(G)$ is called a \textit{dominating set} if every vertex $V (G\setminus D)$ is adjacent to at least
one vertex of $D$. The \textit{domination number} $\gamma(G)$ is the minimum cardinality among the dominating sets of $G$.
\begin{thm}\label{2.13}
Let $M$ be an $R$-module such that every submodule of $M$ contains a minimal submodule and let $\mathcal{M}$ be the set of all minimal submodules of $R$. Then $\mathcal{M}$ is a minimal dominating set of $SSI(M)$ and $\gamma(SSI(M)) \leq |\mathcal{M}|$. Moreover, $\gamma(SSI(M)) = 1$ if and only if $M$ has exactly one minimal submodule or $M$ has exactly two minimal submodules $N_1$ and $N_2$ such that $N_1 + N_2$ is a non-trivial maximal submodule of $M$ such that there is no non-second submodule
properly contained in $N_1 + N_2$. Moreover, if $M$ has exactly two minimal submodules which does
not satisfy the above condition, then $\gamma(SSI(M)) = 2$.
\end{thm}
\begin{proof}
Since any submodule $N$ of $M$ is contains some element $S$ of $\mathcal{M}$ and $N\cap S = S$,
which is a second submodule, $\mathcal{M}$ dominates $SSI(M)$. Let $S \in \mathcal{M}$. It is to be observed that
$\mathcal{M}\setminus \{S\}$ does not dominate $S$ and so fails to dominate $SSI(M)$. Thus $\mathcal{M}$ is a
minimal dominating set of $SSI(M)$ and $\gamma(SSI(M)) \leq |\mathcal{M}|$. The second and third parts follow
from Theorem \ref{2.2}.
\end{proof}

\section{Prime submodule sum graph of a module}
\noindent
\begin{defn}\label{9.1}
Let $M$ be an $R$-module.
The \textit{prime submodule sum graph} of $M$, denoted
by $SSI(M)$, is an undirected simple graph whose vertices are non-zero proper submodules of $M$ and two distinct
vertices $N$ and $K$ are adjacent if and only if $N + K$ is a prime submodule of $M$.
\end{defn}

An $R$-module $M$ is said to be a \emph{multiplication module} if for every submodule $N$ of $M$ there exists an ideal $I$ of $R$ such that $N=IM$ \cite{Ba81}.
\begin{prop}\label{9.7}
Let $M$ be a multiplication $R$-module. Then two distinct non-trivial submodules $N$ and $K$ of $M$ are adjacent in $PSS(M)$ if and only if $(N:_RM)$ and $(K:_RM)$ are adjacent in $PIS(R)$.
\end{prop}
\begin{proof}
This follows from the fact that in the multiplication $R$-module $M$, we have $P$ is a prime submodule of $M$ if and only if
$(P:_RM)$ be a prime ideal of $R$ \cite[Corollary 2.11]{BS89}.
\end{proof}

\begin{thm}\label{9.2}
Let $M$ be an $R$-module such that every submodule of $M$ is contained in a maximal submodule. Then
$PSS(M)$ has a universal vertex if and only if one of the two statements hold:
\begin{itemize}
\item [(a)] $M$ has exactly one maximal submodule.
\item [(b)] $M$ has exactly two maximal submodules $K_1$ and $K_2$ such that $K_1 \cap K_2$ is a
minimal submodule and that there is no non-prime submodule that properly containing in $K_1\cap K_2$.
\end{itemize}
\end{thm}
\begin{proof}
Let (a) hold and $K$ be the maximal submodule of $M$. Then for each submodule $N$ of $M$, we have $N + K = K$, which is a prime submodule and hence $N$ is adjacent to $K$. Thus, $M$ is a universal vertex.

Let (b) hold and set $N:= K_1\cap  K_2$. Assume that $T$ is a non-trivial submodule other than $N$
and without loss of generality, let $T\subseteq K_1$. Since by assumption, $N$ is a minimal submodule, we have
$N\subset N + T\subseteq K_1$. Now, by assumption, $N+ K$ is a prime submodule of $M$ and so $N$ is a universal vertex.

Conversely, let $PSS(M)$ have a universal vertex, say $N$. If $K$ has a unique maximal
submodule, the proof is done. Now, assume that $M$ has at least three maximal submodules, say
$K_1,K_2$ and $K_3$. Note that $N$ cannot be a maximal submodule as two distinct maximal
submodules are not adjacent. Since $N$ is not a maximal submodule, then it is contained in a maximal submodule, say $K_1$. If possible, let $N \not \subseteq K_2$. Then by maximality of $K_2$, we have $N + K_2= M$. Thus $N$ is not adjacent to $K_2$, a contradiction. Therefore, $N\subseteq K_1\cap K_2\cap K_3$. Now, one of the following two cases holds.

\textbf{Case 1.} Let $N \not= K_1 \cap K_2\cap K_3$. Since $N+ (K_1 \cap K_2\cap K_3)=K_1 \cap K_2\cap K_3$
and $N$ is a universal vertex in $PSS(M)$, we get that $K_1 \cap K_2\cap K_3$ is a prime submodule of $M$.
Thus $K_1 K_2 K_3\subseteq K_1 \cap K_2\cap K_3$ implies that $K_1=K_2=K_3$, which is a contradiction.

\textbf{Case 2.} Let $N = K_1\cap K_2\cap K_3$.  Set $T = K_1\cap K_2$. Since the intersection of two maximal
submodules cannot be a prime submodule, $N + T = T$ is not a prime submodule, i.e.  $N$ is not adjacent to $T$, which contradicts
with our assumption that $N$ is a universal vertex. Therefore, $M$ has exactly two
maximal submodules, say $K_1$ and $K_2$. By the same argument as above, we conclude that
$N = K_1 \cap K_2$. Now, we show that $N$ is a minimal submodule. If possible, let there exists a non-trivial
submodule $K$ of $M$ such that $K\subset K_1 \cap K_2=N$. Then, as $N$ is a universal vertex, $K + N = N = K_1 \cap K_2$ is a
prime submodule of $M$, which is a contradiction. Thus, $N$ is a minimal submodule.
If possible, let $K$ be a non-prime submodule such that $N = K_1 \cap K_2 \subset K \subset K_i$, where
$i = 1$ or $2$. But, it follows that $N + K = K$, a non-prime submodule and hence $N$ is not adjacent to $K$, a
contradiction. Thus, there does not exist such submodule $K$ and the proof is completed.
\end{proof}

Recall that a submodule $N$ of an $R$-module $M$ is said to be \textit{small} in $M$ if $N+L \not=M$ for every proper submodule $L$ of $M$ \cite{AF74}. Also, $M$ is said to be a \textit{hollow module} if every proper submodule of $M$ is small \cite{FL74}.

\begin{prop}\label{9.8}
Let $M$ be an $R$-module. If $M$ is a hollow $R$-module such that every proper submodule of $M$ is prime, then $PSS(M)$ is a complete graph.
\end{prop}
\begin{proof}
This is straightforward
\end{proof}

\begin{defn}
We defined the subgraph $\widetilde{SSI(M)}$ of the graph $PIS(R)$ whose vertices are $Ann_R(N)$ for all submodules $N$ of $M$ with $Ann_R(N)$ is a non-trivial ideal of $R$ and
two distinct vertices $Ann_R(N)$ and $Ann_R(K)$ are adjacent if and only if $Ann_R(N)+Ann_R(K)$ is a prime ideal of $R$.
 \end{defn}

\begin{thm}\label{9.9}
Let $M$ be an $R$-module. If $M$ is a faithful finitely generated multiplication $R$-module such that $PSS(M)$ is a complete garph, then $\widetilde{SSI(M)}$ is a complete graph.
\end{thm}
\begin{proof}
Let $M$ be a finitely generated multiplication $R$-module with $PSS(M)$ is a complete garph. Assume that $N$ and $K$ are two  submodules of $M$ such that $Ann_R(N)$ and $Ann_R(K)$ are non-trivial ideals of $R$. Then we get that  $Ann_R(N)M$ and $Ann_R(K)M$ are two non-trivial submodules of $M$. So $Ann_R(N)M+Ann_R(N)M$ is a prime submodule of $M$.
Hence $(Ann_R(N)M+Ann_R(N)M:_RM)=Ann_R(N)+Ann_R(K)$ is a prime ideal of $R$.
\end{proof}
An $R$-module $M$ is said to be a \emph{reduced module} if $rm = 0$ implies that $rM \cap Rm = 0$, where $r\in R$ and $m \in  M$ \cite{MR2050725}.

The intersection of all prime submodules of an $R$-module $M$ is said to be the \emph{prime radical} of $M$ and denote by $rad(M)$. In case $M$ does not contains any prime submodule, the prime radical of $M$ is defined to be $M$ \cite{Lu89}.
\begin{rem}\label{9.3}
Let $M$ be a non-reduced $R$-module such that every submodule of $M$ is contained in a maximal submodule of $M$.
Then $rad(M)$ is adjacent to each prime submodule of $M$.
\end{rem}
\begin{proof}
Since $M$ be a non-reduced $R$-module such that every submodule of $M$ is contained in a maximal submodule of $M$, we get that $rad(M)$ is a non-zero proper submodule of $M$.
For each prime submodule $P$ of $M$, we have $P+ rad(M)=P$ is a prime submodule of $M$. Hence,  $rad(M)$ is adjacent to each prime submodule of $M$.
\end{proof}

\begin{cor}\label{9.4}
Let $M$ be a non-reduced $R$-module such that every submodule of $M$ is contained in a maximal submodule of $M$.
Then $rad(M)$ is the only maximal submodule of $M$ if and only if $rad(M)$ is a universal vertex of $PSS(M)$.
\end{cor}
\begin{proof}
This follows from Theorem \ref{9.2}.
\end{proof}

\begin{thm}\label{9.5}
Let $M$ be an $R$-module such that every submodule of $M$ is contained in a maximal submodule of $M$. Then
a submodule $N$ of $M$ is an isolated vertex in $PSS(M)$ if and only if $N$ is
a maximal as well as minimal submodule of $M$.
\end{thm}
\begin{proof}
Let $N$ be an isolated vertex in $PSS(M)$. If $N$ is not a maximal submodule, then it is
contained in a maximal submodule, say $K$ and $N + K = K$. Thus $N$ is adjacent to $K$ in $PSS(M)$,
a contradiction, and so $N$ is a maximal submodule. If $N$ is not a minimal submodule, then there
exists a submodule $T$ of $M$ such that $0\not=T \subset N$ and $N+ T =N$, which is a
prime submodule. Hence $N$ is adjacent to $T$, a contradiction, and so $N$ is a minimal submodule of $M$.

Conversely, let $N$ be a maximal as well as minimal submodule of $M$. If possible, assume
that $N$ is not isolated in $PSS(M)$. Then there exists a non-zero proper submodule $T$ of $M$ other than $N$ such that $N+ T$ is a  prime submodule. If $T \not \subseteq N$, then as $N$ is maximal submodule of $M$, we
have $N+ T =M$, which is not a prime submodule. On the other hand if $T \subseteq N$, then by minimality
of $N$, we have $T =0$ or $N = T$, a contradiction. Thus such a submodule $T$ does not
exist and hence $N$ is an isolated vertex in $PSS(M)$.
\end{proof}

\begin{thm}\label{9.6}
Let $M$ be an $R$-module such that every submodule of $M$ is contained in a maximal submodule of $M$. Then
we have the following.
\begin{itemize}
\item [(a)] If $PSS(M)$ is a complete graph, then $M$ has exactly one maximal submodule and every proper
non-prime submodule is a minimal submodule.
\item [(b)] If $M$ has exactly one maximal submodule, then $PSS(M)$ is a complete graph.
\end{itemize}
\end{thm}
\begin{proof}
(a) Let $PSS(M)$ be a complete graph. Then by Theorem \ref{9.2}, $PSS(M)$ has a universal vertex and
hence either one of the two conditions holds. Since two distinct maximal
submodules cannot be adjacent, $M$ cannot have two maximal submodules. Hence part (a) of
Theorem \ref{9.2} holds. Let $N$ be a non-zero proper submodule of $M$ which is not prime. If
possible, there exists a submodule $T$ of $M$ such that $0\not= T \subset N$, then $N+ T = N$. Since
$N$ is not a prime submodule, $N$ is not adjacent to $T$, a contradiction to the completeness of $PSS(M)$. Thus
every proper non-prime submodule is a minimal submodule.

(b) Conversely, let $M$ has exactly one maximal submodule, say $K$. Since by assumption, every submodule of $M$ is contained in a maximal submodule, we have $K$ is contains every proper submodule of $M$. Now, as every maximal submodule of $M$ is a prime submodule of $M$, we have $PSS(M)$ is a complete graph.
\end{proof}

\begin{thm}\label{9.7}
Let $M$ be an $R$-module such that every submodule of $M$ is contained in a maximal submodule of $M$. Then $PSS(M)$ is connected if and only if $M$ is not a
direct sum of two of its maximal submodules. If $PSS(M)$ is connected, then diam$(PSS(M)) \leq 2$.
\end{thm}
\begin{proof}
Let $N_1$, $N_2$ be two non-zero submodules of $M$. If $N_1+ N_2$ is a prime submodule, then $N_1- N_2$
in $PSS(M)$. Otherwise, assume that $N_1+ N_2$ is not a prime submodule. By assumption, both submodules $N_1$ and $N_2$ are contained in a maximal submodules of $M$. If they are contained in the same maximal submodule, say $K$, then we have $N_1 - K - N_2$, as maximal submodules
are prime and hence $d(N_1, N_2) = 2$. Thus, we assume that they are not contained in the same maximal submodules, $N_1\subset K_1$, $N_2\subset K_2$ and $K_1 \not= K_2$ are two maximal submodules of $M$.
First assume that $K_1\cap K_2 \not=0$. By the maximality of $K_2$, we have $K_2 + N_1=M$.
Then
$$
K_1=K_1 \cap M=K_1 \cap (K_2+N_1)=(N_1\cap K_1)+(K_1\cap K_2)=N_1+(K_1 \cap K_2)
$$
and so $N_1 -  K_1 \cap K_2$. Similarly,
 $K_1 \cap  K_2-N_2$. Thus we have a path of length 2 given
by $N_1 -  K_1 \cap K_2 - N_2$ and so  $d(N_1, N_2) \leq 2$. Hence,  diam$(PSS(M)) \leq 2$. If $K_1\cap  K_2 =0$, then since $K_1$ and $K_2$ are maximal submodules, we have $M=K_1 \oplus K_2$.

Conversely, if $M$ is direct sum of two maximal submodules $K_1$ and $K_2$, then
$K_1\cong M/K_2$ and $K_2\cong M/K_1$. Thus $K_1$ and $K_2$ are also minimal submodules of $M$. This in turn implies that
$PSS(M)$ have two isolated vertices $K_1$ and $K_2$. Thus $PSS(M)$ is not connected.
\end{proof}

\begin{cor}\label{9.8}
Let $M$ be a multiplication $R$-module such that $M$ is not a
direct sum of two of its maximal submodules. Then diam$(PSS(M)) \leq 2$.
\end{cor}
\begin{proof}
Since every submodule of $M$ is contained in a maximal submodule by \cite[Theorem 2.5]{BS89}, the result follows from Theorem \ref{9.7}.
\end{proof}

\begin{cor}\label{9.998}
Let $M$ be an $R$-module such that every submodule of $M$ is contained a maximal submodule. If $PSS(M)$ is a connected graph, then the sum of any two minimal submodules of $M$ is equal to $M$.
\end{cor}
\begin{proof}
Let $PSS(M)$ be a connected graph.
Assume contrary that $K_1$ and $K_2$ are two minimal submodules of $M$ such that $K_1 + K_2=M$.
Since every submodule of $M$ is contained in amaximal submodule, there is a maximal submodule $P_1$ of $M$ such that $K_1 \subseteq P_1$.
Since  $K_1 + K_2=M$, $K_2 \not \subseteq P_1$. Thus as $K_2$ is minimal, $K_2\cap P_1=0$. Hence
$$
K_1=K_1+ 0=K_1 + (K_2\cap P_1)=(K_1 + P_1)\cap (K_1+ K_2)=P_1.
$$
Thus $K_1$ is a maximal and minimal submodule of $M$. Hence by Theorem \ref{9.5}, $K_1$ is an isolated vertex, which is a contradiction.
\end{proof}

\begin{thm}\label{9.9}
Let $M$ be an $R$-module such that $M$ have two non-comparable submodules which are adjacent in $PSS(M)$. Then
girth($PSS(M)$) = 3.
\end{thm}
\begin{proof}
Let $N_1$ and $N_2$ be two non-comparable submodules which are adjacent in $PSS(M)$.
Then $N_1+ N_2$ is a prime submodule of $M$. Since $N_1$ and $N_2$ are non-comparable, then $N_1$, $N_2$, and $N_1+ N_2$ forms a triangle in $SSI(M)$, i.e. girth($PSS(M)$) = 3.
\end{proof}

\begin{cor}\label{9.10}
If $\SS(M)$ is acyclic or girth$(PSS(M)) > 3$, then no two non-comparable
submodules of $M$ are adjacent in $PSS(M)$ and adjacency occurs only in case
of comparable submodules, i.e. for any edge in $PSS(M)$, one of the terminal vertices is a
prime submodule of $M$.
\end{cor}\begin{cor}\label{9.10}
If $\SS(M)$ is acyclic or girth$(PSS(M)) > 3$, then no two non-comparable
submodules of $M$ are adjacent in $PSS(M)$ and adjacency occurs only in case
of comparable submodules, i.e. for any edge in $PSS(M)$, one of the terminal vertices is a
prime submodule of $M$.
\end{cor}

\begin{thm}\label{9.11}
Let $M$ be an $R$-module such that girth($PSS(M)) = n$. Then there exist at least
$\lfloor n/2 \rfloor$ distinct prime submodules in $M$.
\end{thm}
\begin{proof}
By Theorem \ref{9.9}, if two non-comparable submodules are adjacent in $SSI(M)$,
then girth($SSI(M) = 3$ and the sum of those two non-comparable submodules forms a
prime submodule, and hence $M$ contains at least
$\lfloor 3/2 \rfloor=1$ prime submodule. Thus, we assume
that girth$(PSS(M)) > 3$, i.e. by Corollary \ref{9.10}, adjacency occurs only in case of
comparable submodules. Let $N_1 - N_2 -N_3- \cdots - N_n - N_1$ be a cycle of length $n$. First, we
observe that neither $N_1 \subseteq N_2 \subseteq  N_3\subseteq \cdots \subseteq N_n \subseteq N_1$ nor $N_1 \supseteq N_2 \supseteq N_3\supseteq \cdots\supseteq N_n \supseteq N_1$ can hold, as in both the cases all the submodules will be equal. Thus, without loss of
generality, we have $N_1, N_3 \subseteq  N_2$ and $N_2$ is a prime submodule of $M$. Hence, we have the following
two cases:

\textbf{Case I.} Let $N_1, N_3 \subseteq  N_2$ and $N_3 \subseteq N_4$. Then, we have $N_2$, $N_4$ to be prime submodules.

\textbf{Case II.} Let $N_1, N_3 \subseteq  N_2$ and $N_4 \subseteq N_3$. Then, we have $N_2$, $N_3$ to be prime submodules.
In any case, we get at least 2 submodules to be prime in $M$ among $N_1$, $N_2$, $N_3$ and $N_4$.
Continuing in this manner till $N_n$, we get at least
$\lfloor n/2 \rfloor$ submodules which is prime in $M$.
\end{proof}

\begin{cor}\label{9.12}
Let $M$ has $k$ prime submodules. Then $PSS(M)$ is either acyclic or
girth$(PSS(M)) \leq 2k$.
\end{cor}

\begin{thm}\label{9.13}
Let $M$ be an $R$-module such that every submodule of $M$ is contained a maximal submodule and let $\mathcal{M}$ be the set of all maximal submodules of $R$. Then $\mathcal{M}$ is a minimal dominating set of $PSS(M)$ and $\gamma(PSS(M)) \leq |\mathcal{M}|$. Moreover, $\gamma(PSS(M)) = 1$ if and only if $M$ has exactly one maximal submodule or $M$ has exactly two maximal submodules $N_1$ and $N_2$ such that $N_1 \cap N_2$ is a non-trivial minimal submodule of $M$ such that there is no non-prime submodule
properly contains $N_1 \cap N_2$. Moreover, if $M$ has exactly two maximal submodules which does
not satisfy the above condition, then $\gamma(PSS(M)) = 2$.
\end{thm}
\begin{proof}
Since any submodule $N$ of $M$ is contained in some element $P$ of $\mathcal{M}$ and $N+PS = P$,
which is a prime submodule, $\mathcal{M}$ dominates $PSS(M)$. Let $P \in \mathcal{M}$. It is to be observed that
$\mathcal{M}\setminus \{P\}$ does not dominate $P$ and so fails to dominate $PSS(M)$. Thus $\mathcal{M}$ is a
minimal dominating set of $PSS(M)$ and $\gamma(PSS(M)) \leq |\mathcal{M}|$. The second and third parts follow
from Theorem \ref{9.2}.
\end{proof}

\bibliographystyle{amsplain}

\end{document}